\documentclass{amsart}

\usepackage{ifthen}
\usepackage{graphicx}
\usepackage{amssymb}
\usepackage{enumerate}
\usepackage{url}
\newtheorem{anyprop}{Anyprop}[section]

\newtheorem{theorem}[anyprop]{Theorem}

\theoremstyle{definition}

\newtheorem{definition}[anyprop]{Definition}

\newtheorem{remark}[anyprop]{Remark}

\theoremstyle{remark}

\numberwithin{equation}{section}

\usepackage{amscd}
\usepackage{amssymb}
\input xy
\xyoption{all}

\begin{document}
\title[CONTAINMENT-DIVISION RINGS AND DEDEKIND DOMAINS]
{CONTAINMENT-DIVISION RINGS AND NEW CHARACTERIZATIONS OF DEDEKIND DOMAINS}

%\author[Edisson Gallego] {Edisson Gallego}

\author[Danny A. J. G\'omez-Ram\'irez]{Danny Arlen de Jes\'us G\'omez-Ram\'irez}
\author[Juan D. Velez]{Juan D. V\'elez}
\author[Edisson Gallego]{Edisson Gallego}
\address{Vienna University of Technology, Institute of Discrete Mathematics and Geometry,
wiedner Hauptstaße 8-10, 1040, Vienna, Austria.}
\address{Universidad Nacional de Colombia, Escuela de Matem\'aticas,
Calle 59A No 63 - 20, N\'ucleo El Volador, Medell\'in, Colombia.}
\address{University of Antioquia, Calle 67 \# 53-108, Medell\'in, Colombia}
\email{daj.gomezramirez@gmail.com}
\email{jdvelez@unal.edu.co}
\email{egalleg@gmail.com}

%\address{Universidad Nacional de Colombia, Escuela de Matem\'aticas,
%Calle 59A No 63 - 20, N\'ucleo El Volador, Medell\'in, Colombia.}
%\email{egalleg@unal.edu.co}%\thanks{to my parents}

%\subjclass{14A10 (primary), 13A15 (secondary)}

\begin{abstract}
We introduce a new class of commutative rings with unity, namely, the Containment-Division
Rings (CDR-s). We show that this notion has a very exceptional origin, since it was essentially 
co-discovered with the qualitative help of a computer program (i.e. The Heterogeneous Tool Set (HETS)). Besides, we show that in a Noetherian setting, 
the CDR-s are just another way of describing Dedekind domains. Simultaneously, we see that for CDR-s, the Noetherian condition can be replaced by a weaker
Divisor Chain Condition.

\end{abstract}

\maketitle


\smallskip

\newline\indent

\section*{Introduction}

During the last decades we have seen very rapid advances in the development of 
artificial devices which allow us to verify and complete some quite intricate mathematical proofs 
e.g. the four-color theorem \cite{appel}, \cite{appel2}, \cite{gonthier}; the Feit-Thompson odd order theorem \cite{gonthier2}; and the Kepler Conjecture \cite{hales}, \cite{hales2}; among others.

Moreover, in the former cases the essential concepts behind the proofs were developed by mathematicians and a sophisticated computational support was required in order to handle the huge number of cases to be verified. 

On the other hand, the outstanding work of S. Colon on automated theory formation \cite{colton} and the early versions of his program HR \cite{colton3}, were able to re-discover the number-theoretical notion of refactorable number \cite{colton2}. 

Now, in \cite{gomezetal} and \cite{gomez2}, it was shown how a new valuable and interesting mathematical notion was co-invented with the help of the Heterogeneous Tool Set (HETS) \cite{mossakowski}.

Effectively, during the process of expressing in HETS the mathematical concept of a prime ideal of a commutative ring with unity as a formal blending (i.e. colimit) of the two notions of a meta-prime number and an ideal of a commutative ring with unity; a new condition appears involving the containment and the divisibility relations among the ideals of a ring. So, inspired by this condition the notion of \emph{Containment-Division Ring} (CDR) was proposed (see next section) and subsequently it was discovered that such a class of rings are strongly related with Dedekind Domains.

More specifically, in \cite{gomezetal} and \cite{gomez2} a formalization of concepts as theories in a many-sorted first-order logic with proper signatures was used. Besides, the notion of conceptual blending of two input concepts with commonalities codified through a generic space was computed as the colimit of the corresponding `V'-diagram. 

So, the sentence defining the notion of containment-division ring emerged when we consider the axiom defining the upside-down divisibility relation:

\[ (\forall a,b\in Z)(a \lfloor b\leftrightarrow (\exists c\in Z)(a=c*b)),\]

which belongs to one of the input spaces, i.e., the notion of prime element in a quasi-monoid $(Z,*,1)$ (i.e., $*$ is a binary operation with neutral element $1$). Now, the former condition was reinterpreted in the colimit (blend) concept computed by HETS using as a second concept, the notion of ideals over commutative rings with unity, including a sort $G$ for the collection of all ideals. Effectively, the `conceptual' morphisms between the generic and the input spaces induced syntactic replacements `merging' (i.e., identifying) the sort denoting the quasi-monoid $Z$ with the sort denoting the collection of all ideals of a commutative ring with unity $R$; the upside-down divisibility relation with the containment relation between ideals; and the binary operation in $Z$ with the product of ideals. Thus, the resulting axiom in the blended concept was of the form

\[(\forall a,b \in G)[(a \subseteq b)\leftrightarrow (\exists c\in G)(a=c\cdot_{\iota}b)],\]

where $G$ denotes the set (sort) of ideals of $R$, and $\cdot_{\iota}$ denotes the product of ideals.

Now, the former axiom was exactly the kind of `surprising condition', which we baptized as the \emph{containment-division condition}, allowing us to discover a new class of commutative rings with unity that we will study in the next section.

 %So, this particular mathematical concept is a very unique example of a notion that was discovered with the help of a computer program and, it later turns out that this concept has mathematical interest on its own.

\section{Containment-Division Rings}

\begin{definition}
Let $R$ denote a commutative ring with unity. We say that $R$ is a
\emph{Containment-Division Ring (CDR)}, if for any two ideals $I$ and $J$, it holds
that $I\subseteq J$ if and only if $J$ divides $I$, i.e., if there exists an
ideal $H$ such that $I=HJ$.
\end{definition}

Now, it is straightforward to see that principal ideal domains are, in fact,
CDR-s. Besides, in the setting of integer domains and as it was stated in a 
more informal context in \cite{gomez2}, we will prove formally in this
section that the Noetherian CDR-s are very close related to the Dedekind
domains, i.e., integral domains with the additional property that every
proper ideal can be written as a finite product of ideals \cite[Theorem 37.1
and 37.8]{coleman}. At the same time, we will see that for Noetherian CDR-s
the ascending chain condition can be re-written in the form of a \emph{%
Divisor Chain Condition} (DiCC), i.e., a commutative ring with unity
fulfills the DiCC if for any chain of ideals $I_{1}\subseteq I_{2}\subseteq
\cdots \subseteq I_{n}\subseteq \cdots $ such that $I_{j+1}$ divides $I_{j},$
the chain is stationary.

The formal statement is the following:

\begin{theorem}
\label{CDRDedekind} The following two conditions are equivalent for an
integral domain $R$:

\begin{enumerate}
\item $R$ is a dedekind domain.

\item $R$ is a noetherian CDR.

\item $R$ is a CDR fulfilling the Divisor Chain Condition.
\end{enumerate}
\end{theorem}

\begin{proof}
First let us prove that (1) is equivalent to (2): On the one hand, it is a
clear fact that Dedekind domains are Noetherian \cite[Theorem 37.1]{coleman}%
. Furthermore, it is a classical result that Dedekind Domains fulfill the
Containment-Division Condition (\cite[Fundamental Theorem of OAK-s]{weiss}).
In fact, if we consider two ideals $I$ and $J$ of $R$, if $I\subseteq J$,
then by factoring each ideal, and then localizing on the primes appearing on
their factorizations one obtains the desired division condition \cite[%
Theorem 37.11]{coleman}.

On the other hand, let us consider a proper ideal $I$ of $R$. The case where 
$I$ is a prime ideal is clear. Otherwise, let $P_{1}$ be a prime ideal of $R$
such that $I\subseteq P_{1}$. Then, there exists a proper ideal $I_{1}$ such
that $I=I_{1}P_{1}$, because $R$ is a CDR. If $I_{1}$ is a prime ideal, then
we can clearly express $I$ as a product of two ideals. Otherwise, let us
choose again a prime ideal $P_{2}$ containing $I_{1}$. Hence, there is
another proper ideal $I_{2}$ such that $I_{1}=I_{2}P_{2}$. In the case that $%
I_{2}$ is prime, we can express $I=I_{2}P_{2}P_{1}$ as a finite product of
prime ideals. Otherwise, we could continue inductively in a similar manner.
If for some $r$, the ideal $I_{r}$ is prime, we can write $I$ as a finite
product of prime ideals. If not, one obtains an ascending chain of ideals 
\begin{equation*}
I_{1}\subseteq I_{2}\subseteq I_{3}\subseteq \cdots \subseteq I_{n}\subseteq
\cdots 
\end{equation*}%
Because $R$ is Noetherian, this sequence stops at some point (i.e., there
exists some $m\in \mathbb{N}$ such that for all $i\geq m$, $I_{i}=I_{m}$).
Moreover, $I_{m}=I_{m+1}P_{m+1}=I_{m}P_{m+1}$ and so $I_{m}=I_{m}^{i}P_{m+1}%
\subseteq I_{m}^{i}$, for all $i\in \mathbb{N}$.

Therefore, $I$ and $I_{m}$ are contained in the intersection of the powers
of $I_{m}$, $\cap _{i\geq 1}I_{m}^{i}$. Now, by Krull's Intersection Theorem 
\cite[Corollary 5.4]{eisenbud}, this intersection must be zero. So, $I=(0)$,
which is a prime ideal. In conclusion, $R$ is a Dedekind domain.

Finally, we verify that conditions (2) and (3) are equivalent:

In fact, due to the Containment-Division condition, any ascending chain of
ideals $I_{1}\subseteq I_{2}\subseteq \cdots \subseteq I_{n}\subseteq \cdots 
$ is a divisor chain of ideals and vice versa. So, the ascending chain
condition defining Noetherian rings, which in general implies the DiCC,
turns out to be equivalent to the last one in the setting of CDR-s.
\end{proof}
 
\begin{remark}
The Divisor Chain Condition involved in the third numeral above has also emerged as a natural consequence of 
replacing the containment relation with the divisibility one within the Ascending Chain Condition defining a Noetherian
 ring. So, this intermediate notion between the classes of Noetherian and Non-Noetherian rings can be seen as a natural human-program discovery, since the seminal idea of identifying, in a common setting, the former relations between ideals (i.e., containment and divisibility) was first (co-) suggested by HETS in the context of \cite{gomezetal} and \cite{gomez2} as mentioned before. 
In conclusion, the former theorem is one of the very exceptional instances of a mathematical result whose involved concepts were co-discovered with the qualitative help of a computer program, and which has, simultaneously, enough mathematical value on its own. 
%A very natural conjecture emerges in this context, i.e., is the Ascending Chain Condition, in general,
% stronger than the Divisor Chain Condition? In other words, do commutative rings with unity, 
%which satisfied the DCC but are not Noetherian exists?
\end{remark}

\section*{acknowledgements}
Danny Arlen de Jes\'us G\'omez-Ram\'irez was funded firstly by the European Commission (SP7-ICT-2013-10), FET-Open Grant number: 611553, and subsequently he was also supported by the Vienna Science and Technology Fund (WWTF) as part of the Vienna Research Group 12-004. Finally, he wants to thanks Rafael Betancur, Eduar Cata\~no, Miller Gonzalez and Juan Otalvaro for the support during all these years.

\end{document}